\documentclass[12pt,reqno]{amsart}
\usepackage{amsmath,amssymb,amsfonts,amsthm}
\usepackage[left=3cm, right=3cm, top=2.8cm, bottom=2.8cm]{geometry}
\usepackage[utf8]{inputenc}
\usepackage[T1]{fontenc}
\usepackage{color}
\usepackage{multirow}% http://ctan.org/pkg/multirow
\usepackage{graphicx}% http://ctan.org/pkg/graphicx
\usepackage{hyperref}
\usepackage{stackrel}
\newtheorem{theorem}{Theorem}

\newtheorem{proposition}[theorem]{Proposition}

\newtheorem{definition}[theorem]{Definition}
\newtheorem{remark}[theorem]{Remark}

\begin{document}

\title[Abstract Differential Equations and Caputo Fractional Derivative]{Abstract Differential Equations and Caputo Fractional Derivative}

%%%%%%%%%%%%%%%%%%%%%%%%%%%

\author[P. M. Carvalho-Neto]{Paulo M. de Carvalho-Neto}
\address[Paulo M. de Carvalho Neto]{Departamento de Matem\'atica, Centro de Ciências Físicas e Matemáticas, Universidade Federal de Santa Catarina, Florian\'{o}polis - SC, Brazil}
\email[]{paulo.carvalho@ufsc.br}

%%%%%%%%%%%%%%%%%%%%%%%%%%%


%%%%%%%%%%%%%%%%%%%%%%%%%%%


%%%%%%%%%%%%%%%%%%%%%%%%%%%

\begin{abstract}
%%%
In this work I consider the abstract Cauchy problems with Caputo fractional time derivative of order $\alpha\in(0,1]$, and discuss the continuity of the respective solutions regarding the parameter $\alpha$. I also present a study about the continuity of the Mittag-Leffler families of operators (for $\alpha\in(0,1]$), induced by sectorial operators.
%%%
\end{abstract}

\maketitle

\section{Contextualization of the Main Problem}
\label{sec1}

Let me begin by presenting the abstract Cauchy problem
  \begin{equation}\label{3.1}\tag{$P_\alpha$}
   \left\{\begin{array}{l}
   cD^\alpha_tu(t)=Au(t)+f(t,u(t)),\,\,\,t\geq 0,\\
   u(0)=u_0\in X,
   \end{array} \right.
  \end{equation}
where $X$ is a Banach space over $\mathbb{C}$, $\alpha\in(0,1]$, $A:{D}(A)\subset X\rightarrow X$ is a sectorial operator, $f:[0,\infty)\times X\rightarrow X$ is a continuous function and $cD^{\alpha}_{t}$ denotes Caputo fractional derivative of order $\alpha$, when $\alpha\in(0,1)$, and denotes the standard derivative when $\alpha=1$. (see Section \ref{sec2} for more details).

In a previous work \cite{AnCaCaMa1} (see also \cite{Car1}), which was done with some collaborators, I have presented several results about the existence, uniqueness and continuation of mild solutions to problem \eqref{3.1}, when $\alpha\in(0,1)$. In fact, the results obtained in \cite{AnCaCaMa1} are similar to the classical results that are obtained when $\alpha=1$. Therefore, in order to clarify the differences between these two situations, let me present the following notions and results.

 \begin{definition} Consider $\alpha\in(0,1]$ and $T$ a positive real value.\vspace*{0.2cm}
    \begin{itemize}
    \item[(a)] A function $\phi_\alpha:[0,T]\rightarrow X$ is called a local mild solution of \eqref{3.1} in $[0,T]$ if $\phi_\alpha\in C([0,T];X)$ and
    \begin{equation*}\phi_\alpha(t)=\left\{\begin{array}{ll}E_{\alpha}(At^{\alpha})u_{0}+\displaystyle\int_{0}^{t}(t-s)^{\alpha-1}E_{\alpha,\alpha}(A(t-s)^{\alpha})f(s,\phi_\alpha(s))ds,&\\&\hspace*{-1.5cm}\textrm{if }\alpha\in(0,1),\vspace*{0.2cm}\\
    e^{At}u_{0}+\displaystyle\int_{0}^{t}e^{A(t-s)}f(s,\phi_\alpha(s))ds,&\hspace*{-1.5cm}\textrm{if }\alpha=1,\end{array}\right. \end{equation*}
    for every  $t\in [0,T].$\vspace*{0.2cm}
    \item[(b)] A function $\phi:[0,T)\rightarrow X$ is called a local mild solution of \eqref{3.1} in $[0,T)$, if for any $T^*\in(0,T)$, $\phi(t)$ is a local mild solution of \eqref{3.1} in $[0,T^*]$.\vspace*{0.2cm}
    \end{itemize}
  \end{definition}

 \begin{definition} Let $\alpha\in(0,1]$ be given. A function $\phi:[0,\infty)\rightarrow X$ is called a global mild solution of \eqref{3.1}, if for any $T^*>0$, $\phi(t)$ is a local mild solution of \eqref{3.1} in $[0,T^*]$.
  \end{definition}

It worths to recall that the family of operators $\{e^{At}:t\geq0\}\subset\mathcal{L}(X)$ is the standard analytic semigroup generated by the sectorial operator $A$, while the families $\{E_{\alpha}(At^{\alpha}):t\geq0\}\subset\mathcal{L}(X)$ and $\{E_{\alpha,\alpha}(At^{\alpha}):t\geq0\}\subset\mathcal{L}(X)$, which are called Mittag-Leffler operators, are standard from the literature of fractional differential equations (see for instance \cite{AnCaCaMa1}, and references therein). However, for the completeness of this manuscript I also present more details about these operators in Section \ref{sec2}.

By introducing the notions of continuation of mild solution and maximal local mild solution, it is possible to prove an interesting result, which was called in \cite{AnCaCaMa1} ``Blow up Alternative''. The ``Blow up Alternative'' was what motivated the discussion proposed in this work.

\begin{definition} Let $\alpha\in(0,1]$ and let $\phi_\alpha:[0,{T})\rightarrow X$ be a local mild solution in $[0,{T})$ of problem \eqref{3.1}. If ${T^*}\geq{T}$ and ${\phi_\alpha}^*:[0,{T^*}]\rightarrow X$ is a local mild solution to \eqref{3.1} in $[0,{T^*}]$, with ${\phi_\alpha}^*(t)=\phi_\alpha(t)$, for every $t\in[0,T)$, then I say that ${\phi_\alpha}^*(t)$ is a continuation of $\phi_\alpha(t)$ over $[0,{T^*}]$.
\end{definition}

\begin{definition} Consider $\alpha\in(0,1]$ and $\phi_\alpha:[0,T)\rightarrow X$. If $\phi(t)$ is a local mild solution of \eqref{3.1} in $[0,T)$ that does not have a continuation in $[0,T]$, then I call it a maximal local mild solution of \eqref{3.1} in $[0,T)$.
\end{definition}

\begin{theorem}[Blow up Alternative] \label{bualt} Let $\alpha\in(0,1)$ and $f:[0,\infty)\times X\rightarrow X$ be a continuous function, locally Lipschitz in the second variable, uniformly with respect to the first variable, and bounded (i.e. it maps bounded sets onto bounded sets). Then problem \eqref{3.1} has a global mild solution $\phi_\alpha(t)$ in $[0,\infty)$ or there exist $\omega_\alpha\in (0,\infty)$ such that $\phi_\alpha:[0,\omega_\alpha)\rightarrow X$ is a maximal local mild solution in $[0,\omega_\alpha)$, and in such a case, $\limsup_{t\rightarrow{\omega_\alpha}^{-}}\|\phi_\alpha(t)\|=\infty$.
\end{theorem}

\begin{remark}\label{bualt2}  Although in \cite{AnCaCaMa1} I have not presented a proof for Theorem \ref{bualt} in the case $\alpha=1$, it is simple to adapt such result. Therefore, since this proof does not bring any gain to the results discussed in this paper, I prefer to omit it from this manuscript.
\end{remark}

Theorem \ref{bualt} is my main concern from now on. Observe that for each $\alpha\in(0,1)$ there exists $\omega_\alpha>0$ (it may occurs $\omega_\alpha=\infty$ for some values of $\alpha\in(0,1]$) and a respective maximal local mild solution (or global mild solution) $\phi_\alpha:[0,\omega_\alpha)\rightarrow X$ to problem \eqref{3.1}.

It is very natural that several questions can be raised from the facts emphasized above. In this paper, the main questions I discuss are the following:
{\begin{itemize}
\item[{Q1})] Does $E_{\alpha}(A t^\alpha)$ converges to the semigroup $e^{At}$, when $\alpha\rightarrow 1^{-}$, in some sense?
\item[{Q2})] Does $E_{\alpha,\alpha}(A t^\alpha)$ converges to the semigroup $e^{At}$, when $\alpha\rightarrow 1^{-}$, in some sense?
\item[{Q3})] Is the intersection of every $[0,\omega_\alpha)$ non trivial, i.e.,
\begin{equation*}\{0\}\subsetneq  \stackrel[{\alpha\in(0,1]}]{}{\bigcap}[0,\omega_\alpha)\,\,?\end{equation*}
\item[{Q4})] Does $\phi_\alpha(t)$ converges to $\phi_1(t)$, when $\alpha\rightarrow1^-$, in some sense?\vspace*{0.2cm}
\end{itemize}}
The set of questions made above are called here the ``Limit Problems''. The main objective of this work is to answer these questions as fully as possible.

In order to describe the subjects dealt with in this work, I present a short summary of each of the following sections.

Section \ref{sec2}  makes a small survey about the notions used in this manuscript as a whole. To be more precise, in this section I recall the notions of sectorial operators, analytical semigroups, Mittag-Leffler operators and Caputo fractional derivative of order $\alpha\in(0,1)$.

Finally, in Section \ref{sec3} I present the results that answer the four questions called the Limit Problems, while in Section \ref{sec4} I present a short last discussion about the results obtained in Section \ref{sec3}.

\section{A Small Survey on Sectorial Operators and Fractional Calculus}
\label{sec2}

Let me begin by recalling the notions of sectorial operators and analytical semigroups. It worths to emphasize that for a complete discussion of these notions I may refer to \cite{He1,HiPh1,Ka2,Pa1} and the references therein.

\begin{definition} A semigroup of linear operators on $X$, or simply semigroup, is a family $\{T(t):t\geq0\}\subset\mathcal{L}(X)$ such that:
    \begin{itemize}
    \item[i)] $T(0)=I_X$, where $I=I_X$ is the identity operator in $X$.\vspace*{0.2cm}
    \item[ii)] It holds that $T(t)T(s)=T(t+s)$, for any $t,s\geq0$.\vspace*{0.2cm}
    \end{itemize}
  \end{definition}

   \begin{definition} Assume that $\{T(t):t\geq0\}$ is a semigroup. If it holds that
   $$\lim_{t\rightarrow0^+}T(t)x=x,$$
   for every $x\in X$, then I say that it is a $C_0$-semigroup.
   \end{definition}

\begin{definition} Let $\{T(t):t\geq0\}$ be a $C_0$-semigroup. Its Infinitesimal generator is the linear operator $A:{D}(A)\subset X\to X$, where
  $${D}(A):=\left\{x\in X:\lim_{t\rightarrow0^+}\dfrac{T(t)x-x}{t}\,\textrm{ exists}\right\}$$
  and
  $$Ax:= \lim_{t\rightarrow0^+}\dfrac{T(t)x-x}{t}.$$
  \end{definition}

From the definitions above, I may present the classical result that relates semigroups and the linear part of \eqref{3.1}, when $\alpha=1$.

  \begin{theorem}\label{semigrupo1} Suppose that $\{T(t):t\geq0\}\subset\mathcal{L}(X)$ is a $C_0$-semigroup on $X$. If $A:{D}(A)\subset X\to X$ is the infinitesimal generator of $\{T(t):t\geq0\}$, then $A$ is a closed and densely defined linear operator. Also, for any $x\in {D}(A)$ the application $[0,\infty)\ni t\to T(t)x\in X$ is continuously differentiable and
    $$\dfrac{d}{dt}T(t)x=AT(t)x=T(t)Ax,\qquad \forall t>0.$$
  \end{theorem}

  \begin{proof} See Theorems 10.3.1 and 10.3.3 in \cite{HiPh1}.
  \end{proof}

\begin{remark} From now on, to avoid confusion in this manuscript, whenever $A$ is the infinitesimal generator of $\{T(t):t\geq0\}$, I shall write that $e^{At}$ instead of $T(t)$.
\end{remark}

Now I describe an important curve in the complex plane, that is used throughout the text. %This path was first used to study the gamma function in the complex plane.

  %%%%%%%%%%%%%%%%%%%%%%%%%%%%%%%%%%%%%%%%%%%%%%%%%%%%%%%%%%%%%%%%%%%%%%%%%%%%%%%%%%%%%%%%%%%%%%%%%%%%%%%%%%%%%%%%%%%%%%%%%%%%%%%%%%%%%%%%%%%%%%%

  \begin{definition} The symbol $Ha$ denotes the Hankel's path, if there exist $\epsilon>0$ and $\theta\in(\pi/2,\pi)$, where $Ha=Ha_1+Ha_2-Ha_3$, and the paths $Ha_i$ are given by
  $$Ha=\left\{\begin{array}{l}Ha_1:=\{te^{i\theta}:t\in[\epsilon,\infty)\},\\
  Ha_2:=\{\epsilon e^{it}:t\in[-\theta,\theta)\},\\
  Ha_3:=\{te^{-i\theta}:t\in[\epsilon,\infty)\}.\end{array}\right.$$
  %
  %We also write $Ha=Ha(\epsilon,\theta)$ to show the angle and the radius dependence.
  \end{definition}

 The following discussion provides a brief description of the basic results of the theory of sectorial operators and analytical semigroups.

  %%%%%%%%%%%%%%%%%%%%%%%%%%%%%%%%%%%%%%%%%%%%%%%%%%%%%%%%%%%%%%%%%%%%%%%%%%%%%%%%%%%%%%%%%%%%%%%%%%%%%%%%%%%%%%%%%%%%%%%%%%%%%%%%%%%%%%%%%%%%%%%

  \begin{definition} Let $A:{D}(A)\subset X\to X$ be a closed and densely defined operator. The operator $A$ is said to be a {sectorial operator} if there exist $\theta\in(\pi/2,\pi)$ and $M_\theta>0$ such that
  $$ S_{\theta}:=\{\lambda\in\mathbb{C}\,:\,|\arg(\lambda)|\leq\theta,\,\lambda\not=0\}\subset \rho(A)$$
  and
  $$\|(\lambda-A)^{-1}\|_{\mathcal{L}(X)}\leq\dfrac{M_\theta}{|\lambda|},\,\,\,\,\forall\,\lambda\in S_{\theta}.$$
  \end{definition}

  \begin{theorem}\label{1.103} If $A:{D}(A)\subset X\to X$ is a sectorial operator, then $A$ generates a $C_0$-semigroup $\{e^{At}:t\geq0\}$, which is given by
  $$ e^{At}=\dfrac{1}{2\pi i}\int_{Ha}{e^{\lambda t}(\lambda-A)^{-1}}\,d\lambda,\quad\forall t\geq0,$$
  where $Ha\subset \rho(A)$. Moreover, there exists $M_1>0$ such that
  $$\|e^{At}\|_{\mathcal{L}(X)}\leq M_1,\quad \|Ae^{At}\|_{\mathcal{L}(X)}\leq (M_1/t),$$
  for every $t>0$. Finally, for any $x\in X$ the function $[0,\infty)\ni t \mapsto e^{At}x\in X$ is analytic and
  $$\dfrac{d}{dt}e^{At}x=Ae^{At}x,$$
  for every $t>0$.
  \end{theorem}

  \begin{proof} See Theorem 1.3.4 in \cite{He1} or Theorem 6.13 in \cite[Chapter 2]{Pa1}.
  \end{proof}

 Now, let me draw our attention to the fractional calculus, the Mittag-Leffler functions and their relations with the sectorial operators.

\begin{definition} For $\alpha\in(0,1)$ and $f:[0,T]\rightarrow{X}$, the Riemann-Liouville fractional integral of order $\alpha>0$ of $f(t)$ is defined by
\begin{equation}\label{rlfrac1}J_{t}^\alpha f(t):=\dfrac{1}{\Gamma(\alpha)}\displaystyle\int_{0}^{t}{(t-s)^{\alpha-1}f(s)}\,ds,\end{equation}
for every $t\in [0,T]$ such that integral \eqref{rlfrac1} exists. Above $\Gamma$ denotes the classical Euler's gamma function.

\end{definition}

\begin{definition} For $\alpha\in(0,1)$ and $f:[0,T]\rightarrow{X}$, the Caputo fractional derivative of order $\alpha\in(0,1)$ of $f(t)$ is defined by
\begin{equation}\label{rlfrac2}cD_{t_0,t}^\alpha f(t):=\dfrac{d}{dt}\Big\{J_{t}^{1-\alpha} \big[f(t)-f(0)\big]\Big\},\end{equation}
for every $t\in [0,T]$ such that \eqref{rlfrac2} exists.
\end{definition}

\begin{remark} I could have defined Caputo fractional derivative for any $\alpha>0$, however this is not the focus of this work. It also worths to emphasize that $cD_t^1$ is indeed equal to $(d/dt)$.
\end{remark}

I end this section presenting the Mittag-Leffler operators which are related with the solution of the linear part of \eqref{3.1}, when $\alpha\in(0,1)$ (confront the following result with Theorem \ref{1.103}).

\begin{theorem}\label{2.7} Let $\alpha\in (0,1)$ and suppose that $A:{D}(A)\subset X\to X$ is a sectorial operator. Then, the operators
  $$ E_{\alpha}(At^{\alpha}):=\frac{1}{2\pi i}\int_{Ha}e^{\lambda t}\lambda^{\alpha-1}(\lambda^{\alpha}-A)^{-1}d\lambda,\quad t\ge0,$$
  and
  $$E_{\alpha,\alpha}(At^{\alpha}):=\frac{t^{1-\alpha}}{2\pi i}\int_{Ha}e^{\lambda t}(\lambda^{\alpha}-A)^{-1}d\lambda,\quad t\ge0,\vspace*{0.2cm}$$
  where $Ha\subset \rho(A)$, are well defined. Furthermore, it holds that:
  \begin{itemize}
  \item[(i)] $E_{\alpha}(At^{\alpha})x|_{t=0}=E_{\alpha,\alpha}(At^{\alpha})x|_{t=0}=x$, for any $x\in X$;\vspace*{0.2cm}
  \item[(ii)] $E_{\alpha}(At^{\alpha})$ and $E_{\alpha,\alpha}(At^{\alpha})$ are strongly continuous, i.e., for each $x\in X$
  $$\lim_{t\rightarrow0^+}{\|E_\alpha(A t^\alpha)x-x\|}=0\quad\textrm{and}\quad\lim_{t\rightarrow0^+}{\|E_{\alpha,\alpha}(A t^\alpha)x-x\|}=0;$$
  \item[(iii)] There exists a constant $M_2>0$ {\itshape(uniform on $\alpha$)} such that
  $$\sup_{t\ge0}{\|E_{\alpha}(At^{\alpha})\|_{\mathcal{L}(X)}}\le M_2\quad \textrm{and} \quad \sup_{t\ge0}{\|E_{\alpha,\alpha}(At^{\alpha})\|_{\mathcal{L}(X)}}\le M_2;$$
  \item[(iv)] For each $x\in X$, the function $[0,\infty)\ni t\mapsto E_\alpha(A t^\alpha)x$ is analytic and it is the unique solution of
  $$cD_t^\alpha E_\alpha(A t^\alpha)x=AE_\alpha(A t^\alpha)x,\qquad t>0.$$
\end{itemize}
\end{theorem}

\begin{proof} See Theorem 2.1 in \cite{AnCaCaMa1} and Theorem 2.44 in \cite{Car1}.
\end{proof}

\begin{remark}
For a more complete survey on the special functions named Mittag-Leffler, I suggest the very interesting book \cite{GoKiMaRo1}.
\end{remark}

\section{The Limit Problems}
\label{sec3}

Like it was pointed out in Section \ref{sec1}, here I focus my attention in the Limit Problems. At first, I would like to emphasize that the study I am proposing here was already discussed in the literature, however seeking to answer other questions. I can cite as examples two papers, namely \cite{AlFe1,LiZh1}.

For instance, Almeida-Ferreira in \cite{AlFe1} studied this kind of limit questions involving the solutions of the semilinear fractional partial differential equation,
$$\left\{\begin{array}{ll}cD_t^{1+\alpha} u(t,x)=\triangle_xu(t,x)+|u(t,x)|^\rho u(t,x), &(t,x)\in[0,\infty)\times\mathbb{R}^n,\\
u(0,x)=u_0\textrm{ and }\partial_t u(0,x)=0,& x\in\mathbb{R}^n,\end{array}\right.$$
for $\alpha\in(0,1)$, which interpolates the semilinear heat and wave equations. They studied the existence, uniqueness and regularity of solution to this problem in Morrey spaces and then showed that when $\alpha\rightarrow1^-$, the solution of the problem loses its native regularity. On the other hand, in \cite{LiZh1} the authors discuss these kind of limits to $\alpha-$times resolvent families, which are not our main objective here.

\subsection{Questions Q1) and Q2)}

In order to give a complete answer, let me begin by proving the following limit theorems.

\begin{theorem}\label{mittag1limit} Let $A:D(A)\subset X\rightarrow X$ be a sectorial operator and $\{e^{At}:t\geq0\}$ the $C_0-$semigroup generated by $A$. Now assume that $\alpha\in(0,1)$ and consider the Mittag-Leffler family of operators $\{E_\alpha(A t^\alpha):t\geq0\}$. Then, for each $t\geq0$ I have that
$$\lim_{\alpha\rightarrow1^-}{\|E_\alpha(A t^\alpha)-e^{At}\|_{\mathcal{L}(X)}}=0.$$
Moreover, the convergence is uniform on every compact $S\subset(0,\infty)$, i.e.,
$$\lim_{\alpha\rightarrow1^-}\left\{{\max_{s\in S}\big[\|E_\alpha(A s^\alpha)-e^{As}\|_{\mathcal{L}(X)}\big]}\right\}=0.$$
\end{theorem}

\begin{proof} Let me assume that $x\in X$ and $t>0$. Consider $\theta\in(\pi/2,\pi)$, $S_\theta$ the sector associated with the sectorial operator $A$ and $M_\theta>0$ such that
  \begin{equation}\label{secprop}\|(\lambda-A)^{-1}\|_{\mathcal{L}(X)}\leq\dfrac{M_\theta}{|\lambda|},\,\,\,\,\forall\,\lambda\in S_{\theta}.\end{equation}
  Choose $\epsilon>1$ and consider the Hankel's path $Ha=Ha_1+Ha_2-Ha_3$, where
  \begin{multline*}Ha_1:=\{te^{i\theta}:t\in[\epsilon,\infty)\},\
  Ha_2:=\{\epsilon e^{it}:t\in[-\theta,\theta]\},\\
  Ha_3:=\{te^{-i\theta}:t\in[\epsilon,\infty)\}.
  \end{multline*}

The choices made above ensures that $Ha\subset\rho(A)$, therefore observe by Theorem \ref{1.103} and Theorem \ref{2.7} that
\begin{multline*}E_\alpha(A t^\alpha)x-e^{At}x=\dfrac{1}{2\pi i}\displaystyle\int_{Ha}e^{\lambda t}\Big[\lambda^{\alpha-1}(\lambda^{\alpha}-A)^{-1}x-\lambda^{\alpha-1}(\lambda-A)^{-1}x\Big]\,d\lambda\\
+\dfrac{1}{2\pi i}\displaystyle\int_{Ha}e^{\lambda t}\Big[\lambda^{\alpha-1}(\lambda-A)^{-1}x-(\lambda-A)^{-1}x\Big]\,d\lambda,\end{multline*}
what is equivalent to
\begin{multline*}E_\alpha(A t^\alpha)x-e^{At}x\\=\sum_{j=1}^3\left\{\dfrac{1}{2\pi i}\displaystyle\int_{Ha_i}e^{\lambda t}\Big[\lambda^{\alpha-1}(\lambda^{\alpha}-A)^{-1}x-\lambda^{\alpha-1}(\lambda-A)^{-1}x\Big]\,d\lambda\right.\\
\left.+\dfrac{1}{2\pi i}\displaystyle\int_{Ha_i}e^{\lambda t}\Big[\lambda^{\alpha-1}(\lambda-A)^{-1}x-(\lambda-A)^{-1}x\Big]\,d\lambda\right\}.\end{multline*}

Using the first resolvent identity (see \cite[Lemma 6]{DuSc1}) in the first term in the right side of the above equality, I obtain the identity
\begin{multline}\label{newlimite01}E_\alpha(A t^\alpha)x-e^{At}x\\=\dfrac{1}{2\pi i}\left\{\sum_{j=1}^3\left\{\displaystyle\int_{Ha_i}e^{\lambda t}\Big[\lambda^{\alpha-1}[\lambda-\lambda^{\alpha}](\lambda^{\alpha}-A)^{-1}(\lambda-A)^{-1}x\Big]\,d\lambda\right.\right.\\
\left.\left.+\displaystyle\int_{Ha_i}e^{\lambda t}[\lambda^{\alpha-1}-1](\lambda-A)^{-1}x\,d\lambda\right\}\right\}.\end{multline}

Now assume that $S\subset(0,\infty)$ is a compact set and define the real values
$$M_1=\min\{s:s\in S\}\quad\textrm{e}\quad M_2=\max\{\epsilon\cos(s)t:s\in [-\theta,\theta]\textrm{ and }t\in S\}.$$

If $t\in S$, I have that:\vspace*{0,3cm}

\noindent (i) Over $Ha_1$
  \begin{multline*} \int_{Ha_1}e^{\lambda t}\Big\{\lambda^{\alpha-1}[\lambda-\lambda^{\alpha}](\lambda^{\alpha}-A)^{-1}(\lambda-A)^{-1}x\Big\}\,d\lambda\\
  =\int_{\epsilon}^\infty e^{(\tau e^{i\theta}) t}\Big\{\big(\tau e^{i\theta}\big)^{\alpha-1}\Big[\big(\tau e^{i\theta}\big)-\big(\tau e^{i\theta}\big)^{\alpha}\Big]\\\times\Big(\big(\tau e^{i\theta}\big)^{\alpha}-A\Big)^{-1}\Big(\big(\tau e^{i\theta}\big)-A\Big)^{-1}x\Big\}e^{i\theta}\,d\tau,\end{multline*}
and
  \begin{multline*} \displaystyle\int_{Ha_1}e^{\lambda t}[\lambda^{\alpha-1}-1](\lambda-A)^{-1}x\,d\lambda
  \\=\int_{\epsilon}^\infty e^{(\tau e^{i\theta}) t}\Big\{\Big[\big(\tau e^{i\theta}\big)^{\alpha-1}-1\Big]\Big(\big(\tau e^{i\theta}\big)-A\Big)^{-1}x\Big\}e^{i\theta}\,d\tau,\end{multline*}
from which, by considering the sectorial operator estimate \eqref{secprop}, I deduce
  \begin{multline*}\left\|\int_{Ha_1}e^{\lambda t}\Big\{\lambda^{\alpha-1}[\lambda-\lambda^{\alpha}](\lambda^{\alpha}-A)^{-1}(\lambda-A)^{-1}x\Big\}\,d\lambda\right\|\\
  %\leq (M_\theta)^2\|x\|\left[\int_{\epsilon}^\infty \dfrac{e^{[\tau \cos(\theta)] t}\Big|\big(\tau e^{i\theta}\big)^{\alpha}-\big(\tau e^{i\theta}\big)\Big|}{\tau^2}\,d\tau\right]\hspace*{2cm}\\
  \leq (M_\theta)^2\|x\|\left[\int_{\epsilon}^\infty \dfrac{e^{[\tau \cos(\theta)] M_1}\Big|\big(\tau e^{i\theta}\big)-\big(\tau e^{i\theta}\big)^{\alpha}\Big|}{\tau^2}\,d\tau\right]=:\mathcal{J}_1(\alpha)\|x\|,\end{multline*}
for every $t\in S$, and
  \begin{multline*} \left\|\displaystyle\int_{Ha_1}e^{\lambda t}[\lambda^{\alpha-1}-1](\lambda-A)^{-1}x\,d\lambda\right\|
  \\\leq M_\theta\|x\|\left[\int_{\epsilon}^\infty \dfrac{e^{[\tau \cos(\theta)] M_1}\Big|\big(\tau e^{i\theta}\big)^{\alpha-1}-1\Big|}{\tau^2}\,d\tau \right]=:\mathcal{J}_2(\alpha)\|x\|,\end{multline*}
for every $t\in S$.\vspace*{0,3cm}

\noindent (ii)  Over $Ha_2$
  \begin{multline*} \int_{Ha_2}e^{\lambda t}\Big\{\lambda^{\alpha-1}[\lambda-\lambda^{\alpha}](\lambda^{\alpha}-A)^{-1}(\lambda-A)^{-1}x\Big\}\,d\lambda\\
  =\int_{-\theta}^\theta e^{(\epsilon e^{i\tau }) t}\Big\{\big(\epsilon e^{i\tau }\big)^{\alpha-1}\Big[\big(\epsilon e^{i\tau }\big)-\big(\epsilon e^{i\tau }\big)^{\alpha}\Big]\\\times\Big(\big(\epsilon e^{i\tau }\big)^{\alpha}-A\Big)^{-1}\Big(\big(\epsilon e^{i\tau }\big)-A\Big)^{-1}x\Big\}i\epsilon e^{i\tau }\,d\tau,\end{multline*}
and
  \begin{multline*} \displaystyle\int_{Ha_2}e^{\lambda t}[\lambda^{\alpha-1}-1](\lambda-A)^{-1}x\,d\lambda
  \\=\int_{-\theta}^\theta e^{(\epsilon e^{i\tau }) t}\Big\{\Big[\big(\epsilon e^{i\tau }\big)^{\alpha-1}-1\Big]\Big(\big(\epsilon e^{i\tau }\big)-A\Big)^{-1}x\Big\}i\epsilon e^{i\tau }\,d\tau,\end{multline*}
by considering the sectorial operator estimate \eqref{secprop}, I deduce
  \begin{multline*}\left\|\int_{Ha_2}e^{\lambda t}\Big\{\lambda^{\alpha-1}[\lambda-\lambda^{\alpha}](\lambda^{\alpha}-A)^{-1}(\lambda-A)^{-1}x\Big\}\,d\lambda\right\|\\
  \leq \dfrac{(M_\theta)^2\|x\|}{\epsilon}\left[\int_{-\theta}^\theta e^{[\epsilon\cos(\tau)] t}\Big|\big(\epsilon e^{i\tau }\big)-\big(\epsilon e^{i\tau }\big)^{\alpha}\Big|\,d\tau\right]\\
  \leq \dfrac{(M_\theta)^2\|x\|}{\epsilon}\left[\int_{-\theta}^\theta e^{M_2}\Big|\big(\epsilon e^{i\tau }\big)-\big(\epsilon e^{i\tau }\big)^{\alpha}\Big|\,d\tau\right]=:\mathcal{J}_3(\alpha)\|x\|,\end{multline*}
for every $t\in S$, and
  \begin{multline*} \left\|\displaystyle\int_{Ha_2}e^{\lambda t}[\lambda^{\alpha-1}-1](\lambda-A)^{-1}x\,d\lambda\right\|
  \\\leq \dfrac{M_\theta\|x\|}{\epsilon}\left[\int_{-\theta}^\theta e^{M_2}\Big|\big(\epsilon e^{i\tau }\big)^{\alpha-1}-1\Big|\,d\tau\right]=:\mathcal{J}_4(\alpha)\|x\|.\end{multline*}
for every $t\in S$. \vspace*{0,3cm}

\noindent (iii)  Over $Ha_3$ it is analogous to item (i), therefore
  \begin{multline*}\left\|\int_{Ha_1}e^{\lambda t}\Big\{\lambda^{\alpha-1}[\lambda-\lambda^{\alpha}](\lambda^{\alpha}-A)^{-1}(\lambda-A)^{-1}x\Big\}\,d\lambda\right\|\\
  \leq (M_\theta)^2\|x\|\left[\int_{\epsilon}^\infty \dfrac{e^{[\tau \cos(\theta)] M_1}\Big|\big(\tau e^{-i\theta}\big)-\big(\tau e^{-i\theta}\big)^{\alpha}\Big|}{\tau^2}\,d\tau\right]=:\mathcal{J}_5(\alpha)\|x\|,\end{multline*}
and
  \begin{multline*} \left\|\displaystyle\int_{Ha_1}e^{\lambda t}[\lambda^{\alpha-1}-1](\lambda-A)^{-1}x\,d\lambda\right\|
  \\\leq M_\theta\|x\|\left[\int_{\epsilon}^\infty \dfrac{e^{[\tau \cos(\theta)] M_1}\Big|\big(\tau e^{-i\theta}\big)^{\alpha-1}-1\Big|}{\tau^2}\,d\tau\right]=:\mathcal{J}_6(\alpha)\|x\|.\end{multline*}

Therefore, I can estimate \eqref{newlimite01} by
$$\|E_\alpha(A t^\alpha)-e^{At}\|_{\mathcal{L}(X)}\leq\sum_{i=1}^6\mathcal{J}_i(\alpha),$$
for every $t\in S$.

A consequence of the dominated convergence theorem ensures that
 $$\lim_{\alpha\rightarrow 1^-}\mathcal{J}_i(\alpha)=0,$$
for every $i\in\{1,\ldots,6\}$, therefore I deduce that
$$\lim_{\alpha\rightarrow1^-}\left\{{\max_{s\in S}\big[\|E_\alpha(A s^\alpha)-e^{As}\|_{\mathcal{L}(X)}\big]}\right\}=0.$$
This is the second part of the proof. The first part, for the case $t>0$, follows from this second part, while the first part is trivial when $t=0$.
\end{proof}

\begin{theorem}\label{mittag1.5limit} Consider $A:D(A)\subset X\rightarrow X$ a sectorial operator, $\{e^{At}:t\geq0\}$ the $C_0-$semigroup generated by $A$ and $\{E_\alpha(A t^\alpha):t\geq0\}$ the Mittag-Leffler family of operators. Now assume that:
\begin{itemize}
\item[(i)] $\{\alpha_n\}_{n=1}^\infty\subset(0,1]$ and $\alpha\in(0,1]$ is such that
$$\lim_{n\rightarrow\infty}\alpha_n=\alpha;$$
\item[(ii)] $\{\sigma_n\}_{n=1}^\infty\subset[0,\infty)$ such that
$$\lim_{n\rightarrow\infty}\sigma_n=0.$$
\end{itemize}
Then, it holds that
$$\lim_{n\rightarrow\infty}\|E_{\alpha_n}(A (\sigma_n)^{\alpha_n})-E_{\alpha}(A (\sigma_n)^{\alpha})\|_{\mathcal{L}(X)}=0.$$
\end{theorem}

\begin{proof} By following the same process presented in the proof of Theorem \ref{mittag1limit}, I deduce that
\begin{multline*}E_{\alpha_n}(A (\sigma_n)^{\alpha_n})x-E_\alpha(A (\sigma_n)^{\alpha})x\\=\dfrac{1}{2\pi i}\left\{\sum_{j=1}^3\left\{\displaystyle\int_{Ha_i}e^{\lambda \sigma_n}\Big[\lambda^{\alpha_n-1}[\lambda^{\alpha}-\lambda^{\alpha_n}](\lambda^{\alpha_n}-A)^{-1}(\lambda^\alpha-A)^{-1}x\Big]\,d\lambda\right.\right.\\
\left.\left.+\displaystyle\int_{Ha_i}e^{\lambda \sigma_n}[\lambda^{\alpha_n-1}-\lambda^{\alpha-1}](\lambda^\alpha-A)^{-1}x\,d\lambda\right\}\right\}.\end{multline*}

Last identity allows me to obtain the inequality
$$\|E_\alpha(-A t^\alpha)-e^{At}\|_{\mathcal{L}(X)}\leq\sum_{i=1}^6\mathcal{J}_i(\alpha,\alpha_n),$$
for every $t\in S$, where
\begin{multline*}\mathcal{J}_1(\alpha,\alpha_n)=(M_\theta)^2\left[\int_{\epsilon}^\infty \dfrac{e^{[\tau \cos(\theta)] \sigma_n}\Big|\big(\tau e^{i\theta}\big)^{\alpha}-\big(\tau e^{i\theta}\big)^{\alpha_n}\Big|}{\tau^2}\,d\tau\right],\\
\mathcal{J}_2(\alpha,\alpha_n)=M_\theta\left[\int_{\epsilon}^\infty \dfrac{e^{[\tau \cos(\theta)] \sigma_n}\Big|\big(\tau e^{i\theta}\big)^{\alpha_n-1}-\big(\tau e^{i\theta}\big)^{\alpha-1}\Big|}{\tau^2}\,d\tau \right],\end{multline*}
\begin{multline*}\mathcal{J}_3(\alpha,\alpha_n)=\dfrac{(M_\theta)^2}{\epsilon}\left[\int_{-\theta}^\theta e^{M_2}\Big|\big(\epsilon e^{i\tau }\big)^{\alpha}-\big(\epsilon e^{i\tau }\big)^{\alpha_n}\Big|\,d\tau\right],\\\mathcal{J}_4(\alpha,\alpha_n)=\dfrac{M_\theta}{\epsilon}\left[\int_{-\theta}^\theta e^{M_2}\Big|\big(\epsilon e^{i\tau }\big)^{\alpha_n-1}-\big(\epsilon e^{i\tau }\big)^{\alpha-1}\Big|\,d\tau\right],
\end{multline*}
\begin{multline*}\mathcal{J}_5(\alpha,\alpha_n)=(M_\theta)^2\left[\int_{\epsilon}^\infty \dfrac{e^{[\tau \cos(\theta)] \sigma_n}\Big|\big(\tau e^{-i\theta}\big)^{\alpha}-\big(\tau e^{-i\theta}\big)^{\alpha_n}\Big|}{\tau^2}\,d\tau\right],\\
\mathcal{J}_6(\alpha,\alpha_n)=M_\theta\left[\int_{\epsilon}^\infty \dfrac{e^{[\tau \cos(\theta)] \sigma_n}\Big|\big(\tau e^{-i\theta}\big)^{\alpha_n-1}-\big(\tau e^{-i\theta}\big)^{\alpha-1}\Big|}{\tau^2}\,d\tau \right].\end{multline*}

Again, the dominated convergence theorem ensures that
 $$\lim_{n\rightarrow\infty}\mathcal{J}_i(\alpha,\alpha_n)=0,$$
for every $i\in\{1,\ldots,6\}$, therefore I obtain
$$\lim_{n\rightarrow\infty}\|E_{\alpha_n}(A (\sigma_n)^{\alpha_n})-E_{\alpha}(A (\sigma_n)^{\alpha})\|_{\mathcal{L}(X)}=0.$$
\end{proof}

\begin{theorem}\label{mittag2limit} Let $A:D(A)\subset X\rightarrow X$ be a sectorial operator and $\{e^{At}:t\geq0\}$ the $C_0-$semigroup generated by $A$. Now assume that $\alpha\in(0,1)$ and consider the Mittag-Leffler family of operators $\{E_{\alpha,\alpha}(A t^\alpha):t\geq0\}$. Then, for each $t\geq0$ I have that
$$\lim_{\alpha\rightarrow1^-}{\|E_{\alpha,\alpha}(A t^\alpha)-e^{At}\|_{\mathcal{L}(X)}}=0.$$
Moreover, the convergence is uniform on every compact $S\subset(0,\infty)$, i.e.,
$$\lim_{\alpha\rightarrow1^-}\left\{{\max_{s\in S}\big[\|E_{\alpha,\alpha}(A s^\alpha)-e^{As}\|_{\mathcal{L}(X)}\big]}\right\}=0.$$
\end{theorem}

\begin{proof} Like I did in the proof of Theorem \ref{mittag1limit}, let me assume that $x\in X$ and $t>0$. Consider $\theta\in(\pi/2,\pi)$, $S_\theta$ the sector associated with the sectorial operator $A$ and $M_\theta>0$ such that
  \begin{equation}\label{2secprop}\|(\lambda-A)^{-1}\|_{\mathcal{L}(X)}\leq\dfrac{M_\theta}{|\lambda|},\,\,\,\,\forall\,\lambda\in S_{\theta}.\end{equation}
  Choose $\epsilon>1$ and consider the Hankel's path $Ha=Ha_1+Ha_2-Ha_3$, where
  \begin{multline*}Ha_1:=\{te^{i\theta}:t\in[\epsilon,\infty)\},\
  Ha_2:=\{\epsilon e^{it}:t\in[-\theta,\theta]\},\\
  Ha_3:=\{te^{-i\theta}:t\in[\epsilon,\infty)\}.
  \end{multline*}

The choices made above ensures that $Ha\subset\rho(A)$, therefore observe by Theorem \ref{1.103} and Theorem \ref{2.7} that
$$E_{\alpha,\alpha}(A t^\alpha)x-e^{At}x=\dfrac{1}{2\pi i}\displaystyle\int_{Ha_i}e^{\lambda t}\Big[(\lambda^{\alpha}-A)^{-1}x-(\lambda-A)^{-1}x\Big]\,d\lambda,$$
what, thanks to the first resolvent identity (see \cite[Lemma 6]{DuSc1}), is equivalent to
\begin{equation*}E_{\alpha,\alpha}(A t^\alpha)x-e^{At}x=\dfrac{1}{2\pi i}\sum_{j=1}^3\left\{\displaystyle\int_{Ha_i}e^{\lambda t}[\lambda-\lambda^{\alpha}](\lambda^{\alpha}-A)^{-1}(\lambda-A)^{-1}x\,d\lambda\right\}.\end{equation*}

Now consider $S\subset(0,\infty)$ a compact set and
$$M_1=\min\{s:s\in S\}\quad\textrm{e}\quad M_2=\max\{\epsilon\cos(s)t:s\in [-\theta,\theta]\textrm{ and }t\in S\}.$$

If $t\in S$, I have that:\vspace*{0,3cm}

\noindent (i) Over $Ha_1$
  \begin{multline*} \int_{Ha_1}e^{\lambda t}\Big\{[\lambda-\lambda^{\alpha}](\lambda^{\alpha}-A)^{-1}(\lambda-A)^{-1}x\Big\}\,d\lambda\\
  =\int_{\epsilon}^\infty e^{(\tau e^{i\theta}) t}\Big\{\Big[\big(\tau e^{i\theta}\big)-\big(\tau e^{i\theta}\big)^{\alpha}\Big]\Big(\big(\tau e^{i\theta}\big)^{\alpha}-A\Big)^{-1}\Big(\big(\tau e^{i\theta}\big)-A\Big)^{-1}x\Big\}e^{i\theta}\,d\tau,\end{multline*}
from which, by considering the sectorial operator estimate \eqref{2secprop}, I deduce
  \begin{multline*}\left\|\int_{Ha_1}e^{\lambda t}\Big\{[\lambda-\lambda^{\alpha}](\lambda^{\alpha}-A)^{-1}(\lambda-A)^{-1}x\Big\}\,d\lambda\right\|\\
  %\leq (M_\theta)^2\|x\|\left[\int_{\epsilon}^\infty \dfrac{e^{[\tau \cos(\theta)] t}\Big|\big(\tau e^{i\theta}\big)^{\alpha}-\big(\tau e^{i\theta}\big)\Big|}{\tau^2}\,d\tau\right]\hspace*{2cm}\\
  \leq (M_\theta)^2\|x\|\left[\int_{\epsilon}^\infty \dfrac{e^{[\tau \cos(\theta)] M_1}\Big|\big(\tau e^{i\theta}\big)-\big(\tau e^{i\theta}\big)^{\alpha}\Big|}{\tau^{\alpha+1}}\,d\tau\right]=:\mathcal{J}_1(\alpha)\|x\|,\end{multline*}
for every $t\in S$.\vspace*{0,3cm}

\noindent (ii)  Over $Ha_2$
  \begin{multline*} \int_{Ha_2}e^{\lambda t}\Big\{[\lambda-\lambda^{\alpha}](\lambda^{\alpha}-A)^{-1}(\lambda-A)^{-1}x\Big\}\,d\lambda\\
  =\int_{-\theta}^\theta e^{(\epsilon e^{i\tau }) t}\Big\{\Big[\big(\epsilon e^{i\tau }\big)-\big(\epsilon e^{i\tau }\big)^{\alpha}\Big]\Big(\big(\epsilon e^{i\tau }\big)^{\alpha}-A\Big)^{-1}\Big(\big(\epsilon e^{i\tau }\big)-A\Big)^{-1}x\Big\}i\epsilon e^{i\tau }\,d\tau,\end{multline*}
and therefore, if I consider the sectorial operator estimate \eqref{2secprop}, I obtain
  \begin{multline*}\left\|\int_{Ha_2}e^{\lambda t}\Big\{[\lambda-\lambda^{\alpha}](\lambda^{\alpha}-A)^{-1}(\lambda-A)^{-1}x\Big\}\,d\lambda\right\|\\
  \leq (M_\theta)^2\epsilon^{-\alpha}\|x\|\left[\int_{-\theta}^\theta e^{M_2}\Big|\big(\epsilon e^{i\tau }\big)-\big(\epsilon e^{i\tau }\big)^{\alpha}\Big|\,d\tau\right]=:\mathcal{J}_2(\alpha)\|x\|,\end{multline*}
for every $t\in S$. \vspace*{0,3cm}

\noindent (iii)  Over $Ha_3$ it is analogous to item (i), therefore
  \begin{multline*}\left\|\int_{Ha_1}e^{\lambda t}\Big\{[\lambda-\lambda^{\alpha}](\lambda^{\alpha}-A)^{-1}(\lambda-A)^{-1}x\Big\}\,d\lambda\right\|\\
  %\leq (M_\theta)^2\|x\|\left[\int_{\epsilon}^\infty \dfrac{e^{[\tau \cos(\theta)] t}\Big|\big(\tau e^{i\theta}\big)^{\alpha}-\big(\tau e^{i\theta}\big)\Big|}{\tau^2}\,d\tau\right]\hspace*{2cm}\\
  \leq (M_\theta)^2\|x\|\left[\int_{\epsilon}^\infty \dfrac{e^{[\tau \cos(\theta)] M_1}\Big|\big(\tau e^{-i\theta}\big)-\big(\tau e^{-i\theta}\big)^{\alpha}\Big|}{\tau^{\alpha+1}}\,d\tau\right]=:\mathcal{J}_3(\alpha)\|x\|.\end{multline*}

Thus, I achieve the estimate
$$\|E_\alpha(A t^\alpha)-e^{At}\|_{\mathcal{L}(X)}\leq\mathcal{J}_1(\alpha)+\mathcal{J}_2(\alpha)+\mathcal{J}_3(\alpha),$$
for every $t\in S$.

A consequence of the dominated convergence theorem ensures
 $$\lim_{\alpha\rightarrow 1^-}\mathcal{J}_1(\alpha)=0,\quad\lim_{\alpha\rightarrow 1^-}\mathcal{J}_2(\alpha)=0,\quad\lim_{\alpha\rightarrow 1^-}\mathcal{J}_3(\alpha)=0,$$
therefore I have
$$\lim_{\alpha\rightarrow1^-}\left\{{\max_{s\in S}\big[\|E_{\alpha,\alpha}(A s^\alpha)-e^{As}\|_{\mathcal{L}(X)}\big]}\right\}=0.$$
This is the second part of the proof. The first part, for the case $t>0$, follows from this second part, while the first part is trivial when $t=0$.
\end{proof}

%\begin{remark}  Recall that, as presented in \cite[Section 18.1]{Ba1}, it holds that
%%
%\begin{equation}\label{novabate}
%E_{\alpha,\beta}(z)=\frac{1}{\alpha}z^{(1-\beta)/\alpha}e^{z^{1/\alpha}}+\epsilon_{\alpha,\beta}(z),\quad \textrm{if}\quad |\arg(z)|\leq\frac{1}{2}\pi\alpha,
%\end{equation}
%%
%where
%%
%\begin{equation*}
%\epsilon_{\alpha,\beta}(z)=-\sum_{n=1}^{N-1}\frac{z^{-n}}{\Gamma(\beta-\alpha n)}+O(|z|^{-N}),\ \textrm{when}\ z\rightarrow\infty.
%\end{equation*}
%
% Therefore, if $A\in(-\infty,0)$ (which would happen if we assume that the sectorial operator $A$ of Theorem \ref{mittag1limit} is a constant real value) and $S\subset[0,\infty)$, then \eqref{novabate} allow us to deduce that
%%
%$$\lim_{\alpha\rightarrow1^-}\left\{{\max_{s\in S}\Big[|E_{\alpha}(A s^\alpha)-e^{As}|\Big]}\right\}=0\quad\textrm{and}\quad\lim_{\alpha\rightarrow1^-}\left\{{\max_{s\in S}\Big[|E_{\alpha,\alpha}(A s^\alpha)-e^{As}|\Big]}\right\}=0.$$
%%
%
% we have that vTheorems \ref{mittag1limit} and \ref{mittag2limit} seems to be sharp with respect the fact that we cannot prove those theorems is $S\in[0,\infty)$, however I couldn't prove it.
%\end{remark}

A natural question that arises at this point of this manuscript is: Is it possible to obtain other modes of convergence to these operators? To answer this question, I begin by presenting the following notion.

\begin{definition} Let $S\subset\mathbb{R}$ and $p\geq 1$. The symbol $L^{p}(S;\mathcal{L}(X))$ is used to represent the set of all Bochner measurable functions $f:S\rightarrow \mathcal{L}(X)$ for which $\|f\|_{\mathcal{L}(X)}\in L^{p}(S;\mathbb{R})$, where $L^{p}(S;\mathbb{R})$ stands for the classical Lebesgue space. Moreover, $L^{p}(S;\mathcal{L}(X))$ is a Banach space when considered with the norm
$$\|f\|_{L^{p}(S;\mathcal{L}(X))}:=\left[\int_S{\|f(s)\|^p_{\mathcal{L}(X)}}\,ds\right]^{1/p}.$$
The vectorial spaces $L^{p}(S;\mathcal{L}(X))$ are called Bochner-Lebesgue spaces.
\end{definition}

For more details on the Bochner measurable functions and Bochner-Lebesgue integrable functions, I may refer to \cite{ArBaHiNe1,Mik1} and references therein.

Now I present the last theorem of this subsection.

\begin{theorem}\label{intemit} Let $A:D(A)\subset X\rightarrow X$ be a sectorial operator and $\{e^{At}:t\geq0\}$ the $C_0-$semigroup generated by $A$. Now assume that $\alpha\in(0,1)$ and consider the Mittag-Leffler families $\{E_{\alpha}(A t^\alpha):t\geq0\}$ and $\{E_{\alpha,\alpha}(A t^\alpha):t\geq0\}$.
If $p\geq1$ and $S\subset\mathbb{R}^+$ is compact, then
$$e^{At}\in L^{p}(S;\mathcal{L}(X)),\quad E_{\alpha}(A t^\alpha)\in L^{p}(S;\mathcal{L}(X))\quad\textrm{and}\quad E_{\alpha,\alpha}(A t^\alpha)\in L^{p}(S;\mathcal{L}(X)).$$
Moreover, it holds that
\begin{multline*}\lim_{\alpha\rightarrow1^-}\left[\int_S{\|E_{\alpha}(A s^\alpha)-e^{As}\|^p_{\mathcal{L}(X)}}\,ds\right]^{1/p}=0,\\
\lim_{\alpha\rightarrow1^-}\left[\int_S{\|E_{\alpha,\alpha}(A s^\alpha)-e^{As}\|^p_{\mathcal{L}(X)}}\,ds\right]^{1/p}=0.\end{multline*}
\end{theorem}

\begin{proof} Like it was done in the proof of Theorem \ref{mittag1limit}, if I choose $\epsilon>1$ and $\theta\in(\pi/2,\pi)$, with the aid of  Theorem \ref{1.103}, I can deduce the existence of $\widetilde{M}>0$ such that
$$ \|e^{At}\|_{\mathcal{L}(X)}\leq \widetilde{M}\left[\int_{\epsilon}^\infty \dfrac{e^{[\tau \cos(\theta)] t}}{\tau}\,d\tau+\int_{-\theta}^\theta e^{[\epsilon\cos(\tau)] t}\,d\tau\right]$$
for every $t\in S$. Thus, by considering Minkowski's integral inequality (see Theorem 202 in \cite{HaLiPo1}) and the triangular inequality of $L^p(S;X)$, I obtain
\begin{multline*}\hspace*{-0.3cm}\left[\int_S\|e^{At}\|^p_{\mathcal{L}(X)}\,dt\right]^{1/p}\leq\widetilde{M}\left[\int_S\left(\int_{\epsilon}^\infty \dfrac{e^{[\tau \cos(\theta)] t}}{\tau}\,d\tau+\int_{-\theta}^\theta e^{[\epsilon\cos(\tau)] t}\,d\tau\right)^p\,dt\right]^{1/p}\\
\leq\widetilde{M}\left[\int_{\epsilon}^\infty\left(\dfrac{1}{\tau}\right)\left(\int_S {e^{p[\tau \cos(\theta)] t}}\,dt\right)^{1/p}\,d\tau+\int_{-\theta}^\theta\left(\int_S {e^{p[\epsilon \cos(\tau)] t}}\,dt\right)^{1/p}\,d\tau\right].
\end{multline*}

Define $M_1=\min\{s:s\in S\}\geq0$ and $M_2=\max\{s:s\in S\}\geq0$. Since $\cos(\theta)<0$ and $\cos(\tau)\leq1$, if $\tau\in[-\theta,\theta]$, I can estimate the right side of the above inequality by
$$\widetilde{M}\left[\int_{\epsilon}^\infty\left(\dfrac{1}{\tau}\right)\left(\int_{M_1}^{M_2} {e^{p[\tau \cos(\theta)] t}}\,dt\right)^{1/p}\,d\tau+\int_{-\theta}^\theta|S|^{1/p} {e^{\epsilon M_2}}\,d\tau\right],
$$
in order to obtain
\begin{multline*}\left[\int_S\|e^{At}|^p_{\mathcal{L}(X)}\,dt\right]^{1/p}\\\leq
\widetilde{M}\left[\left\{\int_{\epsilon}^\infty\dfrac{e^{[\tau \cos(\theta)] M_1}}{\big[p\tau^{p+1}(-\cos(\theta))\big]^{1/p}}\,d\tau\right\}+2\theta|S|^{1/p} {e^{\epsilon M_2}}\right]<\infty,
\end{multline*}
i.e., $e^{A t}\in L^{p}(S;\mathcal{L}(X))$.

Now, bearing in mind Theorem \ref{2.7} and following the exact same steps done above, I can deduce that
\begin{multline*}\left[\int_S\|E_{\alpha}(At^\alpha)|^p_{\mathcal{L}(X)}\,dt\right]^{1/p}\\\leq
\widetilde{M}\left[\left\{\int_{\epsilon}^\infty\dfrac{e^{[\tau \cos(\theta)] M_1}}{\big[p\tau^{p+1}(-\cos(\theta))\big]^{1/p}}\,d\tau\right\}+2\theta|S|^{1/p} {e^{\epsilon M_2}}\right]<\infty,
\end{multline*}
and
\begin{multline*}\left[\int_S\|E_{\alpha,\alpha}(At^\alpha)|^p_{\mathcal{L}(X)}\,dt\right]^{1/p}\\\leq
\widetilde{M}\left[\left\{\int_{\epsilon}^\infty\dfrac{e^{[\tau \cos(\theta)] M_1}}{\big[p\tau^{\alpha p+1}(-\cos(\theta))\big]^{1/p}}\,d\tau\right\}+2\theta|S|^{1/p} {e^{\epsilon M_2}}\right]<\infty.
\end{multline*}
i.e., $E_{\alpha}(A t^\alpha),E_{\alpha,\alpha}(A t^\alpha)\in L^{p}(S;\mathcal{L}(X))$.

Now, let me prove the convergence. Like it was done in the proof of Theorem \ref{mittag1limit}, I deduce the existence of $\widetilde{M}>0$ such that
\begin{multline*}\|E_\alpha(-A t^\alpha)x-e^{At}x\|\leq\widetilde{M}\|x\|\left[\int_{\epsilon}^\infty \dfrac{e^{[\tau \cos(\theta)] t}\Big|\big(\tau e^{i\theta}\big)-\big(\tau e^{i\theta}\big)^{\alpha}\Big|}{\tau^2}\,d\tau\right.\\
\left.+\int_{-\theta}^\theta e^{M_2}\Big|\big(\epsilon e^{i\tau }\big)^{\alpha-1}-1\Big|\,d\tau+\int_{\epsilon}^\infty \dfrac{e^{[\tau \cos(\theta)] t}\Big|\big(\tau e^{-i\theta}\big)-\big(\tau e^{-i\theta}\big)^{\alpha}\Big|}{\tau^2}\,d\tau\right],\end{multline*}
for every $t\in S$.

At first, observe that the triangular inequality of $L^p(S;X)$ ensures that
$$\left(\int_S\|E_\alpha(-A t^\alpha)-e^{At}\|^p_{\mathcal{L}(X)}\,dt\right)^{1/p}\\\leq \widetilde{M}\Big[\mathcal{I}_\alpha(t)+\mathcal{J}_\alpha+\mathcal{K}_\alpha(t)\Big],$$
where
\begin{multline*}\mathcal{I}_\alpha(t)=\left(\int_S\left[\int_{\epsilon}^\infty \dfrac{e^{[\tau \cos(\theta)] t}\Big|\big(\tau e^{i\theta}\big)-\big(\tau e^{i\theta}\big)^{\alpha}\Big|}{\tau^2}\,d\tau\right]^p dt\right)^{1/p},\\\mathcal{J}_\alpha=|S|^{1/p}e^{M_2}\left[\int_{-\theta}^\theta \Big|\big(\epsilon e^{i\tau }\big)^{\alpha-1}-1\Big|\,d\tau \right],\\
\textrm{ and }\quad
\mathcal{K}_\alpha(t)=\left(\int_S\left[\int_{\epsilon}^\infty \dfrac{e^{[\tau \cos(\theta)] t}\Big|\big(\tau e^{-i\theta}\big)-\big(\tau e^{-i\theta}\big)^{\alpha}\Big|}{\tau^2}\,d\tau\right]^p dt\right)^{1/p}.\end{multline*}

Like above, by applying Minkowski's integral inequality (see Theorem 202 in \cite{HaLiPo1}) I deduce that
$$\mathcal{I}_\alpha(t)\leq \int_{\epsilon}^\infty\left[\left(\int_{M_1}^{M_2} e^{[p\tau \cos(\theta)] t}dt\right)^{1/p}\dfrac{\Big|\big(\tau e^{i\theta}\big)-\big(\tau e^{i\theta}\big)^{\alpha}\Big|}{\tau^2}\right]\,d\tau,$$
what implies in the estimate
$$
\mathcal{I}_\alpha(t)\leq \int_{\epsilon}^\infty\dfrac{e^{[\tau \cos(\theta)] M_1}\Big|\big(\tau e^{i\theta}\big)-\big(\tau e^{i\theta}\big)^{\alpha}\Big|}{\big[p\tau^{2p+1}\big(-\cos(\theta)\big)\big]^{1/p}}\,d\tau=:\mathcal{I}_\alpha.$$
Similarly, I obtain
$$\mathcal{K}_\alpha(t)\leq \int_{\epsilon}^\infty\dfrac{e^{[\tau \cos(\theta)] M_1}\Big|\big(\tau e^{-i\theta}\big)^{\alpha}-\big(\tau e^{-i\theta}\big)\Big|}{\big[p\tau^{2p+1}\big(-\cos(\theta)\big)\big]^{1/p}}\,d\tau=:\mathcal{K}_\alpha.$$

Again, as a consequence of the dominated convergence theorem I deduce that
 $$\lim_{\alpha\rightarrow 1^-}\mathcal{I}_\alpha=0,\quad\lim_{\alpha\rightarrow 1^-}\mathcal{J}_\alpha=0\quad\textrm{and}\quad\lim_{\alpha\rightarrow 1^-}\mathcal{K}_\alpha=0,$$
therefore
$$\lim_{\alpha\rightarrow1^-}\left(\int_S\|E_\alpha(-A t^\alpha)-e^{At}\|^p_{\mathcal{L}(X)}\,dt\right)^{1/p}=0.$$

By analogous arguments I also conclude that
$$\lim_{\alpha\rightarrow1^-}\left(\int_S\|E_{\alpha,\alpha}(-A t^\alpha)-e^{At}\|^p_{\mathcal{L}(X)}\,dt\right)^{1/p}=0.$$
\end{proof}

\begin{remark} Observe that for $S\subset (0,\infty)$ the convergence presented in the above result becomes a consequence of Theorems \ref{mittag1limit} and \ref{mittag2limit}. However, since I am considering $S\subset\mathbb{R}^+$, the above result is in fact more general.
\end{remark}

\subsection{Questions Q3) and Q4)}
In this subsection I begin to deal with the nonlinearity $f(t,x)$, introducing new computations to the results obtained so far. Recall that if $f(t,x)$ is a continuous function, locally Lipschitz in the second variable, uniformly with respect to the first variable, and bounded (i.e. it maps bounded sets onto bounded sets), Theorem \ref{bualt} and Remark \ref{bualt2} ensure that for any $\alpha\in(0,1]$, there exists $w_\alpha>0$ (which can be $\infty$) and a unique maximal local mild solution (or global mild solution) $\phi_\alpha:[0,\omega_\alpha)\rightarrow X$ of the Cauchy problem \eqref{3.1}.

However there is no simple way to know if
\begin{equation}\label{3.61045}\stackrel[{\alpha\in(0,1]}]{}{\bigcap}[0,\omega_\alpha)=\{0\},\end{equation}
which, if this is the case, would lead me to face the possibility of calculate the fractional limit of the family of functions $\{\phi_\alpha(t):\alpha\in(0,1]\}$, when $\alpha\rightarrow1^-$, just to the constant value $\phi_\alpha(0)=u_0$. This is exactly what question Q3) intends to discuss.

Let me begin by presenting the following result.

\begin{theorem}\label{3.61045} Assume that $f:[0,\infty)\times X\rightarrow X$ is a continuous function, locally Lipschitz in the second variable, uniformly with respect to the first variable, and bounded (i.e. it maps bounded sets onto bounded sets), and consider problem \eqref{3.1}, for $\alpha\in(0,1]$. If $\phi_\alpha(t)$ is the maximal local mild solution of \eqref{3.1} defined over the interval $[0,\omega_{\alpha})$  (which can be $(0,\infty)$ for some values of $\alpha$), then for each $\alpha_0\in(0,1)$, there exists $t_{\alpha_0}>0$ such that
$$[0,t_{\alpha_0}]\subset\stackrel[{\alpha\in[\alpha_0,1]}]{}{\bigcap}[0,\omega_{\alpha}).$$
\end{theorem}

\begin{proof} Observe that for each $\alpha\in(0,1]$, Theorem \ref{bualt} and Remark \ref{bualt2} ensure the existence and uniqueness of a maximal local mild solution (or global mild solution) of \eqref{3.1}, such that the maximal interval of existence for the solution $\phi_\alpha(t)$ is $[0,\infty)$ or a finite interval $[0,\omega_\alpha)$ and, in such case,
\begin{equation}\label{blownow1}\limsup_{t\rightarrow\omega_\alpha^{-}}\|\phi_\alpha(t)\|=\infty.\end{equation}

Since $f(t,x)$ is locally Lipschitz in the second variable, there exist $R_0,L_0>0$ such that
$$\|f(t,x)-f(t,y)\|\leq L_0\|x-y\|$$
for every $(t,x),(t,y)\in B\big((0,u_0),R_0\big)$.

Let $\alpha_0\in(0,1)$ and suppose that
$$\stackrel[{\alpha\in[\alpha_0,1]}]{}{\bigcap}[0,\omega_{\alpha})=\{0\}.$$
Then, there should exist a sequence $\{\alpha_n\}_{n=1}^{\infty}\subset[\alpha_0,1]$ and a related decreasing sequence of positive real numbers $\{\omega_{\alpha_n}\}_{n=1}^{\infty}$, such that $\lim_{n\rightarrow\infty}{\omega_{\alpha_n}}=0$. By choosing a subsequence of $\{\omega_{\alpha_n}\}_{n=1}^{\infty}$, if necessary, I can suppose that there exist $\alpha\in[\alpha_0,1]$ such that $\alpha_n$ converges to $\alpha$ and a sequence $\{\sigma_n\}_{n=1}^\infty$ with the following properties:
\begin{itemize}
\item[(i)] $0<\omega_{\alpha_{n+1}}<\sigma_n<\omega_{\alpha_n}$,\, for every $n\in\mathbb{N}$;\vspace*{0.2cm}
\item[(ii)] $\|\phi_{\alpha_n}(t)-u_0\|<R_0/2$, for $t\in[0,\sigma_n)$, and $\|\phi_{\alpha_n}(\sigma_n)-u_0\|=R_0/2$, \, for every $n\in\mathbb{N}$. This assertion follows mainly from \eqref{blownow1}.
\end{itemize}

Now observe that
\begin{multline*}\|E_{{\alpha_n}}(A(\sigma_n)^{{\alpha_n}})u_0-u_0\|\leq\|E_{{\alpha_n}}(A(\sigma_n)^{{\alpha_n}})u_0-E_{{\alpha}}(A(\sigma_n)^{{\alpha}})u_0\|
\\+\|E_{{\alpha}}(A(\sigma_n)^{{\alpha}})u_0-u_0\|,\end{multline*}
therefore, Theorems \ref{2.7} and \ref{mittag1.5limit} ensure that
$$\lim_{n\rightarrow\infty}\|E_{{\alpha_n}}(A(\sigma_n)^{{\alpha_n}})u_0-u_0\|=0.$$

Thus, choose $N_0\in \mathbb{N}$ such that
$$\sigma_n<R_0/2\qquad\textrm{and}\qquad \|E_{{\alpha_n}}(A(\sigma_n)^{{\alpha_n}})u_0-u_0\|\leq R_0/4,$$
for any $n\geq N_0$. Then, if $n\geq N_0$ I have
\begin{multline*}\|\phi_{\alpha_n}(\sigma_n)-u_0\|\leq\|E_{{\alpha_n}}(A(\sigma_n)^{{\alpha_n}})u_0-u_0\|\\
+\displaystyle\left\|\int_{0}^{\sigma_n}(\sigma_n-s)^{{\alpha_n}-1}E_{{\alpha_n},{\alpha_n}}(A(\sigma_n-s)^{{\alpha_n}})f(s,\phi_{\alpha_n}(s))ds\right\|,\end{multline*}
and therefore
\begin{multline*}\|\phi_{\alpha_n}(\sigma_n)-u_0\|\leq (R_0/4)\\
+M_2\displaystyle\int_{0}^{\sigma_n}(\sigma_n-s)^{{\alpha_n}-1}\Big[\big\|f(s,\phi_{\alpha_n}(s))-f(s,u_0)\big\|+\big\|f(s,u_0)\big\|\Big]\,ds,\end{multline*}
where $M_2>0$ is given uniformly on $\alpha\in[\alpha_0,1]$ by Theorem \ref{2.7}.

Now, since $(\sigma_n,\phi_{\alpha_n}(\sigma_n)),(\sigma_n,u_0)\in B\big((0,u_0),R_0\big)$, I have that
$$\|\phi_{\alpha_n}(\sigma_n)-u_0\|\leq (R_0/4) +M_2M_3\sigma_n
+M_2L_0\displaystyle\int_{0}^{\sigma_n}(\sigma_n-s)^{{\alpha_n}-1}\|\phi_{\alpha_n}(s)-u_0\|\,ds,$$
where $M_3={\max_{s\in[0,\omega_{\alpha_1}]}{\|f(s,0)\|}}.$

Finally, from \cite[Lemma 7.1.1]{He1} I have that
$$\|\phi_{\alpha_n}(\sigma_n)-u_0\|\leq \theta_n+\eta_n\displaystyle\int_{0}^{\sigma_n}(\sigma_n-s)^{{\alpha_n}-1}E_{\alpha_n,\alpha_n}\left(\eta_n(\sigma_n-s)\right)\theta_n\,ds,$$
where $\eta_n=\big[M_2L_0\Gamma(\alpha_n)\big]^{1/\alpha_n}$ and $\theta_n=(R_0/4) +M_2M_3\sigma_n$. Thus, using the dominated convergence theorem and Theorems \ref{2.7}, I deduce that
\begin{multline*}R_0/2=\lim_{n\rightarrow\infty}{\|\phi_{\alpha_n}(\sigma_n)-u_0\|}\\
\leq\lim_{n\rightarrow\infty}\theta_n
+\left(\lim_{n\rightarrow\infty}\eta_n\right)\left[\lim_{n\rightarrow\infty}\displaystyle\int_{0}^{\sigma_n}(\sigma_n-s)^{{\alpha_n}-1}E_{\alpha_n,\alpha_n}\left(\eta_n(\sigma_n-s)\right)\theta_n\,ds\right]
\\=(R_0/4),\end{multline*}
what is a contradiction. This concludes the proof of this theorem.

\end{proof}

In order to give another characterization to the biggest interval contained in the set $\cap_{\alpha\in[\alpha_0,1]}[0,\omega_{\alpha})$, for each $\alpha_0\in(0,1)$, I present the following result.

\begin{theorem}\label{3.610451} Consider same hypotheses of Theorem \ref{3.61045} and assume that for any $\alpha_0$ in $(0,1)$, there exists $\alpha\in[\alpha_0,1]$ such that $\omega_{\alpha}<\infty$. Define
$$\Omega_{\alpha_0}:=\sup\{t>0:[0,t]\subset\cap_{\alpha\in[\alpha_0,1]}[0,\omega_{\alpha})\}.$$
Then, I have that:
\begin{itemize}
\item[(i)] for each $\alpha_0\in(0,1)$ it holds that $\inf\{\omega_\alpha:\alpha\in[\alpha_0,1]\textrm{ with }\omega_\alpha<\infty\}=\Omega_{\alpha_0}$;\vspace*{0.2cm}
\item[(ii)] for each $\alpha_0\in(0,1)$ I have that $[0,\Omega_{\alpha_0})\subset[0,\omega_\alpha)$, for every $\alpha\in[\alpha_0,1]$;\vspace*{0.2cm}
\item[(iii)] if $\alpha_1,\alpha_2\in(0,1]$, with $\alpha_1<\alpha_2$, then $\Omega_{\alpha_1}\leq\Omega_{\alpha_2}$.
\end{itemize}
\end{theorem}

\begin{proof} $(i)$ Since by definition it holds that
$$\Omega_{\alpha_0}\leq \inf\{\omega_\alpha:\alpha\in[\alpha_0,1]\textrm{ with }\omega_\alpha<\infty\},$$
let me assume for an instant that $\Omega_{\alpha_0}< \inf\{\omega_\alpha:\alpha\in[\alpha_0,1]\textrm{ with }\omega_\alpha<\infty\}$. If such is the case, by setting
$$r:=\left[\Omega_{\alpha_0}+ \inf\{\omega_\alpha:\alpha\in[\alpha_0,1]\textrm{ with }\omega_\alpha<\infty\}\right]/2,$$
I could verify that $\Omega_{\alpha_0}+r\in \{t>0:[0,t]\subset\cap_{\alpha\in[\alpha_0,1]}[0,\omega_{\alpha})\}$, what would lead me to the contradiction $\Omega_{\alpha_0}+r\leq\Omega_{\alpha_0}$. Therefore
$$\Omega_{\alpha_0}= \inf\{\omega_\alpha:\alpha\in[\alpha_0,1]\textrm{ with }\omega_\alpha<\infty\},$$
as I wanted.

Finally item $(ii)$ follows from the definition of $\Omega_{\alpha_0}$, while item $(iii)$ is a directly consequence of the infimum properties.
\end{proof}

\begin{remark}\label{3.6104512} It worths to emphasize the importance of the hypotheses: ``for any $\alpha_0$ in $(0,1)$, there exists $\alpha\in[\alpha_0,1]$ such that $\omega_{\alpha}<\infty$''. In fact, if this is not the case, there should exists $\tilde{\alpha}\in(0,1)$ such that for every $\alpha\in[\tilde{\alpha},1]$ I would have
$$[0,\omega_{\alpha})=[0,\infty).$$
\end{remark}

Now, let me change the subject a little and begin to address question Q4).

\begin{theorem}\label{finalalphalimit} Consider the Cauchy problem \eqref{3.1} with $\alpha\in(0,1]$ and assume that function $f:[0,\infty)\times X\rightarrow X$ is continuous, locally Lipschitz in the second variable, uniformly with respect to the first variable, and bounded (i.e. it maps bounded sets onto bounded sets). Then, if $\phi_\alpha(t)$ is the maximum local mild solution (or global mild solution) of \eqref{3.1} defined over $[0,\omega_{\alpha})$ (or $[0,\infty)$), there exists $t_*>0$ such that
$$\lim_{\alpha\rightarrow 1^-}{\|\phi_\alpha(t)-\phi_{1}(t)\|}=0,$$
for every $t\in[0,t^*]$.
\end{theorem}

\begin{proof} Observe that Theorem \ref{bualt} and Remark \ref{bualt2} guarantee the existence of a maximal mild solution (or global mild solution) $\phi_\alpha(r)$ defined over a maximal interval of existence $[0,\omega_{\alpha})$ (or $[0,\infty)$), such that for $\alpha\in(0,1]$ satisfies
\begin{equation*}\phi_\alpha(t)=\left\{\begin{array}{ll}E_{\alpha}(At^{\alpha})u_{0}+\displaystyle\int_{0}^{t}(t-s)^{\alpha-1}E_{\alpha,\alpha}(A(t-s)^{\alpha})f(s,\phi_\alpha(s))ds,&\\&\hspace*{-1.5cm}\textrm{if }\alpha\in(0,1),\vspace*{0.2cm}\\
    e^{At}u_{0}+\displaystyle\int_{0}^{t}e^{A(t-s)}f(s,\phi_\alpha(s))ds,&\hspace*{-1.5cm}\textrm{if }\alpha=1.\end{array}\right. \end{equation*}

Now I need to make two considerations.
\begin{itemize}
\item[(i)] Note that Theorem \ref{3.610451} and Remark \ref{3.6104512} ensure that
$$[0,\Omega_{1/2}/2]\subset \cap_{\alpha\in[1/2,1]}[0,\omega_{\alpha})\quad\textrm{or}\quad[0,\infty)\subset \cap_{\alpha\in[1/2,1]}[0,\omega_{\alpha}).$$
In any case, it is not difficult to choose $t_0>0$ such that $[0,t_0]\subset [0,\omega_{\alpha})$, for every $\alpha\in[1/2,1]$.\vspace*{0.2cm}
\item[(ii)] Recall that $f(r,x)$ is locally Lipschitz in the second variable, i.e., there exist $R_0,L_0>0$ such that
$$\|f(r,x)-f(r,y)\|\leq L_0\|x-y\|$$
for every $(r,x),(r,y)\in B\big((0,u_0),R_0\big)$. In that way, I can infer that there exists $t_*\in[0,t_0]$ such that
$$\cup_{\alpha\in[1/2,1]}\{(t,\phi_\alpha(t)):t\in[0,t_*]\}\subset B\big((0,u_0),R_0/2\big),$$
since otherwise, it would exist $\alpha_*\in[1/2,1]$ and two sequences of real numbers $\{\alpha_n\}_{n=1}^{\infty}\subset[1/2,1]$ and $\{\tau_n\}_{n=1}^\infty\subset [0,t_0]$, with $\alpha_n\rightarrow \alpha_*$ and $\tau_n\rightarrow 0$, when $n\rightarrow\infty$, that also satisfies $\|\phi_{\alpha_n}(t)-u_0\|<R_0/2$, for all $t\in[0,\tau_n)$, and $\|\phi_{\alpha_n}(\tau_n)-u_0\|=R_0/2$.

However this construction, like the one given in the proof of Theorem \ref{3.61045}, leads to a contradiction.\vspace*{0.2cm}
\item[(iii)] Define $N:=\sup\{\|f(s,\phi_1(s))\|:s\in[0,t_*]\}$, which is finite since $f(s,\phi_{1}(s))$ is continuous and $[0,t_*]$ is a compact interval.
\end{itemize}

Taking into account considerations $(i)$, $(ii)$ and $(iii)$, I deduce that
\begin{multline*}\|\phi_{\alpha}(t)-\phi_{1}(t)\|\leq\|E_{\alpha}(At^{\alpha})u_0-e^{At}u_0\|\\
+\left\|\displaystyle\int_{0}^{t}(t-s)^{\alpha-1}E_{\alpha,\alpha}(A(t-s)^{\alpha})\Big[f(s,\phi_\alpha(s))-f(s,\phi_1(s)\Big]ds\right\|\hspace*{3cm}\\
+\left\|\displaystyle\int_{0}^{t}(t-s)^{\alpha-1}\Big[E_{\alpha,\alpha}(A(t-s)^{\alpha})-e^{A(t-s)}\Big]f(s,\phi_1(s))ds\right\|\\
+\left\|\displaystyle\int_{0}^{t}\Big[(t-s)^{\alpha-1}-1\Big]e^{A(t-s)}f(s,\phi_1(s))ds\right\|,\end{multline*}
for every $\alpha\in[1/2,1]$ and $t\in[0,t_*]$, therefore, by considering also Theorems \ref{1.103} and \ref{2.7}, I obtain
\begin{multline*}\|\phi_{\alpha}(t)-\phi_{1}(t)\|\leq\|E_{\alpha}(At^{\alpha})-e^{At}\|_{\mathcal{L}(X)}\|u_0\|\\
+L_0M_2\displaystyle\int_{0}^{t}(t-s)^{\alpha-1}\|\phi_\alpha(s))-\phi_1(s)\|ds\\
\hspace*{6cm}+N\displaystyle\int_{0}^{t}s^{\alpha-1}\Big\|E_{\alpha,\alpha}(As^{\alpha})-e^{As}\Big\|_{\mathcal{L}(X)}ds\\
+NM_1\displaystyle\int_{0}^{t}\big[s^{\alpha-1}-1\big]ds.\end{multline*}
for every $\alpha\in[1/2,1]$  and $t\in[0,t_*]$. Note that I may rewrite the above inequality as
\begin{equation}\label{finalequi}\|\phi_{\alpha}(t)-\phi_{1}(t)\|\leq g_\alpha(t)+L_0M_2\displaystyle\int_{0}^{t}(t-s)^{\alpha-1}\|\phi_\alpha(s))-\phi_1(s)\|ds,\end{equation}
for every $\alpha\in[1/2,1]$  and $t\in[0,t_*]$, where
\begin{multline*}g_\alpha(t)=\|E_{\alpha}(At^{\alpha})-e^{At}\|_{\mathcal{L}(X)}\|u_0\|+N\displaystyle\int_{0}^{t}s^{\alpha-1}\Big\|E_{\alpha,\alpha}(As^{\alpha})-e^{As}\Big\|_{\mathcal{L}(X)}ds
\\+NM_1\displaystyle\int_{0}^{t}\big[s^{\alpha-1}-1\big]ds.\end{multline*}

Finally, by considering \cite[Lemma 7.1.1]{He1}, I have that \eqref{finalequi}, the dominated convergence theorem, Theorems \ref{mittag1limit}, \ref{mittag2limit} and \ref{intemit} allow us to deduce that
$$\lim_{\alpha\rightarrow1^-}\|\phi_{\alpha}(t)-\phi_{1}(t)\|=0,$$
for each $t\in[0,t_*]$. This completes the proof of the theorem.
 \end{proof}

\section{Closing Remarks}\label{sec4} All the computations done in this manuscript could have being more general. I could have considered $\alpha_0\in(0,1]$ and studied the limit
$$\lim_{\alpha\rightarrow\alpha_0}{u_\alpha(t)}=u_{\alpha_0}(t),$$
or, I could have considered the Cauchy problem with initial condition $u(t_0)=u_0$. However, these two more general problems would only bring technical difficulties, since the computations would be analogous. That is why I avoided dealing with these situations here.

As a final commentary, recall that I can consider ``well-behaved'' functions $f(t,x)$ in order to obtain better results.

\begin{proposition}\label{propfinal} If $f:[0,\infty)\times X\rightarrow X$ is globally Lipschitz on the second variable, uniformly with respect to the first variable, that is, there exist $L>0$ such that
$$\|f(t,x)-f(t,y)\|\leq L\|x-y\|$$
for every $x,y\in X$ and every $t\in[0,\infty)$, then $f$ maps bounded sets into bounded sets.
\end{proposition}
\begin{proof} Just observe that for any $t\geq0$ and $x\in X$ I have the inequality
$$\|f(t,x)\|\leq\|f(t,x)-f(t,0)\|+\|f(t,0)\|\leq L\|x\|+\|f(t,0)\|.$$
\end{proof}

Proposition \ref{propfinal} allows me to conclude that every solution of the Cauchy problem \eqref{3.1}, when $f(t,x)$ is globally Lipschitz on the second variable, uniformly with respect to the first variable, is defined in $[0,\infty)$ (this is a consequence of estimatives I have already done in this manuscript, Theorem \ref{bualt}, Remark \ref{bualt2} and \cite[Lemma 7.1.1]{He1}). Therefore, I would be able to compute the limit done in Theorem \ref{finalalphalimit} for every $t\in[0,\infty)$, in order to obtain a more general result.

The discussion about functions $f(t,x)$ that are locally Lipschitz on the second variable, uniformly with respect to the first variable, deserves a more profound study, which is being developed.

\end{document}